%% file: main.tex
\documentclass[letterpaper, 10pt, conference]{ieeeconf}      % Use this line for a4
\IEEEoverridecommandlockouts                              % This command is only
\usepackage{amsmath} % assumes amsmath package installed
\usepackage{amssymb,mathrsfs,color,nameref}  % assumes amsmath package installed

\usepackage{graphicx} % Required for inserting images
\usepackage{cite}

\input{new_commands}

\newcommand{\R}{\mathbb{R}}

\renewenvironment{proof}{{\it Proof.}~~}{\hfill$\square$}

\makeatletter
\renewenvironment{description}
  {\list{}{\labelwidth=0cm \leftmargin=0cm \labelsep=0cm
           }}
  {\endlist}
\makeatother

\title{\LARGE \bf Consensus under Persistence Excitation}
\author{Fabio Ancona, Mohamed Bentaibi and Francesco Rossi% <-this % stops a space
\thanks{F. Ancona and M. Bentaibi are with Dipartimento di Matematica ``Tullio Levi-Civita'', Universit\`a\ degli Studi di Padova, Via Trieste 63, 35121 Padova, Italy
        {\tt\small ancona@math.unipd.it, bentaibi@math.unipd.it}}%
\thanks{F. Rossi is with Dipartimento di Culture del Progetto, Università Iuav di Venezia, 30135 Venezia, Italy
        {\tt\small francesco.rossi@iuav.it}. F. Rossi is a member of G.N.A.M.P.A. (I.N.d.A.M.).}%
}

\begin{document}

\maketitle
\thispagestyle{empty}
\pagestyle{empty}

%%%%%%%%%%%%%%%%%%%%%%%%%%%%%%%%%%%%%%%%%%%%%%%%%%%%%%%%%%%%%%%%%%%%%%%%%%%%%%%%
\begin{abstract}
We prove that a first-order cooperative system of interacting agents converges to consensus if the so-called Persistence Excitation condition holds. This condition requires that the interaction function between any pair of agents satisfies an integral lower bound. The interpretation is that the interaction needs to ensure a minimal amount of service.
\end{abstract}

%%%%%%%%%%%%%%%%%%%%%%%%%%%%%%%%%%%%%%%%%%%%%%%%%%%%%%%%%%%%%%%%%%%%%%%%%%%%%%%%
\section{INTRODUCTION}

Studying  self-organization and emergence of patterns in collective dynamics has gained significant prominence in the control community. In particular, several researches have focused on understanding mechanisms underlying the dynamics of multi-agent systems, such as those enforcing the emergence of \textit{consensus}. In simple terms, this means that all agents reach an agreement, as for instance  a unanimous vote in elections. Applications of such models are found in a wide variety of fields, see e.g. \cite{tomlin1998conflict,jadbabaie2003coordination,benatti2020opinion,krause2012opinion,zha2020opinion,chuang2007state,bellomo2012dynamics}.

In this article, we study first-order cooperative systems of the following form:
\begin{eqnarray}\label{e-ODE-noM}
    \dot{x}_i(t) = \frac{\lambda_i}{N}\sum_{j=1}^{N} \phi_{ij}(t) (x_j(t) - x_i(t))
\end{eqnarray}
for $i \in \{1,\ldots,N\}$ where
\begin{align*}
    \phi_{ij}(t) = \phi(|x_i(t) - x_j(t)|)\geq 0
\end{align*}
for $ i,j \in\{1,\ldots,N\}$. It describes the evolution of $N \geq 2$ agents on an Euclidean space $\R^d$. The position $x_i(t) \in \mathbb{R}^d$ may represent opinion on different topics, velocity or other attributes of agent $i$ at time $t$. The (nonlinear) influence function $\phi_{ij}(t):\R\to\R$ is used to quantify the influence of agent $j$ on agent $i$, where $i,j\in\{1,\ldots,N\}$. The term $\lambda_i$ is a scaling parameter. 

From the modelling point of view, each agent is expected to communicate with its neighbours through a \textit{network topology}, influenced by sensor characteristics and the environment. While the easiest scenario involves a fixed network topology (e.g. \cite{watts1998collective, olfati2007consensus}), practical situations often involve dynamic changes, due to factors like communication dropouts, security concerns, or intermittent actuation. In this setting, potential connection losses between agents occur, hindering reaching consensus. Therefore, when interactions between agents are subject to failure, it becomes crucial to investigate whether consensus can still be achieved or not.  We model this scenario as follows:
\begin{eqnarray}\label{e-ODE}
    \dot{x}_i(t) = \frac{\lambda_i}{N}\sum_{j=1}^{N} M_{ij}(t)\phi_{ij}(t)(x_j(t) - x_i(t))
\end{eqnarray}
for $i \in \{1,\ldots,N\}$. The terms $M_{ij}:[0,+\infty)\mapsto [0,1]$ represent the weight given to the (directed) connection of agent $j$ with agent $i$. They encode the time-varying network topology and account for potential communication failures (e.g., when they vanish).

We quantify the possible lack of interactions by introducing the condition of \textit{persistent excitation} (PE from now on).
\begin{definition}[Persistent excitation]\label{defof:PE} Let $T,\mu>0$ be given. We say that the function $M \in L^{\infty}([0,+\infty);[0,1])$ satisfies the PE condition with parameters $\mu,T$ if it holds
     \begin{equation}\label{defof:PEgeneral}\tag{PE}
         \int_t^{t+T}M(s)\,ds \geq \mu \qquad \forall t \geq 0.
     \end{equation}
\end{definition}
~

Imposing the PE condition means that such a function is not too weak on any given time interval of length $T$. This can be seen as a condition on the minimum level of service.

Although the PE condition is a standard tool in classical control theory (see \cite{narendra2012stable,ChSi2010,ChSi2014}), its use in multi-agent systems has gained interest in the last years (see e.g. \cite{ren2008distributed,tang2020bearing,manfredi2016criterion,bonnet2021consensus}). Many of the results however impose a PE condition either on functions depending on the state of the system (see e.g. \cite{manfredi2016criterion}), or on quantities like the averaged graph-Laplacian generated by the communication weights with
respect to the variance bilinear form (see \cite{bonnet2021consensus}) or the scrambling coefficients generated by the communication weights (see \cite{bonnet2022consensus}). In all of these cases, the PE condition is then applied to quantities which in principle require a regular monitoring of the state of the system, instead of being applied to the communication weights only. In \cite{anderson2016convergence}, for instance, the authors prove that consensus holds by applying the PE condition on symmetric communication weights only ($M_{ij}=M_{ji}$) which are moreover regulated in the sense that one-sided limits exist for all $t \geq 0$, and they consider only the case where $\phi \equiv 1$.

%We assume that the map $\phi(\cdot):\mathbb{R}_+\mapsto \mathbb{R}$ and $M_{ij}:\mathbb{R}_+\mapsto [0,1]$ satisfy the following hypotheses: 

The main result of this article shows that the PE condition on a general class of weights is sufficient to ensure consensus.
\begin{theorem}\label{thm:main}
Let $\{x_i(t)\}_{i=1}^N$ be a solution of \eqref{e-ODE} with initial data $\{\bar x_i\}_{i=1}^N$.

Assume the following conditions:
\begin{description}
\item[(H1) \label{hyp:Lip}] The function $\phi(\cdot):[0,+\infty)\to \mathbb{R}$ is Lipschitz continuous.

\item[(H2) \label{hyp:phimin}] The constant
\begin{eqnarray}\label{defof:phimin}
    \phi_{\min} := \min_{r \in [0,\max_{i,j \in \{1,\ldots N\}}|\bar x_i - \bar x_j|]}\phi(r)
\end{eqnarray}
satisfies $\phi_{\min} > 0$.

\item[(H3) \label{hyp:M}] All weights $M_{ij}:[0,+\infty)\to [0,1]$ are $\mathscr{L}^1$-measurable.

\end{description}

Fix  $T,\mu>0$ and assume that all $M_{ij}$ satisfy \eqref{defof:PEgeneral}. Then, consensus holds:
\begin{align*}
    \lim_{t \rightarrow + \infty} x_i(t) - x_j(t)=0 \quad \forall i,j \in \{1,\ldots,N\}.
\end{align*}

\end{theorem}

~

The main interest of this result is given by the fact that consensus holds even under a very weak PE condition, e.g. when $\mu$ is very small and $T$ very large. This property is cleary connected with the fact that \eqref{e-ODE} is a cooperative system, i.e. interactions are always non-repulsive. Then, even when time runs with no interactions (or very weak ones), the overall configuration does not move far away from consensus.

We also show some simulations of systems of the form \eqref{e-ODE}, see Section \ref{s-sim}. While Theorem \ref{thm:main} ensures convergence for any conditions, simulations allow to estimate the rate of convergence. Unsurprisingly, the average rate of convergence is decreasing as long as $\mu$ decreases. More surprisingly, the decrease is very regular (linear in log-log coordinates) , both for linear and non-linear dynamics.

The structure of the article is the following: in Section \ref{section: models}, we describe two important models for opinion formation; we then provide some general properties of the dynamics. In Section \ref{section: proof of main}, we give the proof of the main result, first in the real line and then in the multidimensional setting. In Section \ref{s-sim}, we show and discuss some simulations in the one-dimensional case. Finally, in Section \ref{section: conclusion}, we draw some conclusions and present future research directions.

\section{Models of opinion formation}\label{section: models}

In this section, we describe two important models for opinion formation of the form \eqref{e-ODE-noM}. 

In the classical case, the function $\phi_{ij}(t)$ is symmetric and $\lambda_i > 0$ are fixed \cite{rainer2002opinion,degroot1974reaching}. In this setting, the average value is preserved and the system is cooperative. If the influence function satisfies condition {\bf \nameref{hyp:phimin}}, then all $x_i$ converge to the average value. However, such a setting has some scaling problems for large number of agents. Indeed, large groups of agents may have strong impact on small groups, even though they are very far from each other. For this reason, a different scaling has been proposed in \cite{motsch2011new}:
\begin{eqnarray*}
    \lambda_i = \frac{N}{\sum_{j=1}^N \phi_{ij}(t)}.
\end{eqnarray*}
This rescaling introduces asymmetry in the dynamics, thus average is not preserved. Yet, also in this setting one reaches consensus under condition {\bf \nameref{hyp:phimin}}.

We treat both cases in a unitary way, from now on. For this reason, we define
\begin{eqnarray}
 \phi_{\max} := \max_{r \in [0,\max_{i,j \in \{1,\ldots N\}}|\bar x_i - \bar x_j|]}\phi(r),\\
\label{defof:lambda_i}
    \lambda_i :=
    \begin{cases}
        1 \quad &\text{for fixed weights}\\
        \frac{N}{\sum_{j=1}^N \phi_{ij}} \quad &\text{for rescaled weights,}
    \end{cases}\\
\label{defof:kmax}
     K_{\max} :=
    \begin{cases}
        \phi_{\max} \quad &\text{for fixed weights}\\
        \frac{\phi_{\max}}{\phi_{\min}} \quad &\text{for rescaled weights,}
    \end{cases}\\
\label{defof:kmin}
     K_{\min} :=
    \begin{cases}
        \phi_{\min} \quad &\text{for fixed weights}\\
        \frac{\phi_{\min}}{\phi_{\max}} \quad &\text{for rescaled weights.}
    \end{cases}
\end{eqnarray}

We use the following inequalities later:
\begin{eqnarray}\label{sum of kernels less than eta max}
    \frac{K_{\min}}{N}\sum_{j=1}^N M_{ij} \leq \frac{\lambda_i}{N}\sum_{j=1}^N M_{ij}\phi_{ij} \leq K_{\max}
\end{eqnarray}
for all $i \in \{1,\ldots, N\}$. They are direct consequences  of the definitions above and the condition $M_{ij} \leq 1$. 

We also later use the following lemma, which proof is a direct computation.
\begin{lemma}\label{l-reverse}
    Given $\{x_1(t), \ldots, x_N(t)\}$ solution of \eqref{e-ODE}, then $\{-x_1(t), \ldots, -x_N(t)\}$ is a solution of \eqref{e-ODE} too.
    \end{lemma}
    The interest of this lemma is that it permits to reverse several results, e.g. properties about the maximum becoming properties about the minimum. 

\subsection{General properties}

In this section, we prove  properties of solutions of \eqref{e-ODE}. We first show that the diameter satisfies a dissipative property. Before doing so, we provide a useful lemma.

\begin{lemma}\label{lemma:scalarproductinequality}
     Let $i,j\in \{1,\ldots,N\}$ be a pair of indices such that
    \begin{align*}
        \max_{k,l \in \{1,\ldots,N\}}|x_k - x_l|=|x_i - x_j|.
    \end{align*}
It then holds
\begin{eqnarray*}
&&    \max_{k \in \{1,\ldots,N\}}\langle x_k, x_i - x_j\rangle = \langle x_i, x_i - x_j\rangle \\
 &&   \min_{k \in \{1,\ldots,N\}}\langle x_k, x_i - x_j\rangle = \langle x_j, x_i - x_j\rangle .
\end{eqnarray*}
\end{lemma}
\begin{proof}
    The proof is entirely similar to the proof \cite[Lemma 3.4]{bonnet2022consensus} in the context of graphons.
\end{proof}
\begin{proposition}\label{prop: contract of support}
        
    The function 
 \begin{eqnarray*}
     \gamma_{\max}(t):= \max_{i \in \{1,\ldots,N\}}\left|x_i(t)\right| 
 \end{eqnarray*}
 is non-increasing

 Similarly, the size of the support 
 \begin{eqnarray*}
     \mathcal{D}(t):= \max_{i,j \in \{1,\ldots,N\}}\left|x_i(t)-x_j(t)\right| 
 \end{eqnarray*}
 is non-increasing
 
Similarly, on the real line, the function 
 \begin{eqnarray*}
     \gamma_{\min}(t):= \min_{i \in \{1,\ldots,N\}}x_i(t)
 \end{eqnarray*}
 is non-decreasing.
\end{proposition}
\begin{proof}
The functions $\gamma_{\max}$ and $\mathcal{D}$ are Lipschitz, since they are the pointwise maxima of a finite family of Lipschitz continuous functions. By Rademacher's theorem, they are differentiable almost everywhere. By Dankin's theorem (see \cite{danskin2012theory}) it  holds
  \begin{eqnarray*}
       \frac{1}{2}\frac{d}{dt}\gamma_{\max}^2(t) = \max_{i \in \Pi_1(t)}\left\langle \frac{d}{dt}x_i(t),x_i(t)\right\rangle
   \end{eqnarray*}
 where $\Pi_1(t) \in \{1,\ldots,N\}$ represents the nonempty subset of indices for which the maximum is reached. Fix an arbitrary $p\in \Pi_1(t)$ and observe that for all $j \in \{1,\ldots,N\}$ it  holds
\begin{eqnarray*}
    \left\langle x_{p}-x_j, x_{p}\right\rangle \geq 0,
\end{eqnarray*}
 which implies that for all $t \geq 0$ it holds
 \begin{eqnarray*}
     &\left\langle \frac{d}{dt}x_p(t),x_p(t)\right\rangle \\
     &= \frac{\lambda_p}{N}\sum_{j=1}^{N} M_{pj}(t)\phi_{pj}(t) \left\langle x_j(t) - x_p(t), x_p(t) \right\rangle \leq  0.
 \end{eqnarray*}
 Since this estimate holds for any $p \in \Pi_1(t)$, we have
\begin{eqnarray*}
    \frac{d}{dt}\gamma_{\max}(t) \leq 0 \quad \forall t \geq 0,
\end{eqnarray*}
i.e. the function $\gamma_{\max}$ is non-increasing. 

The statement for the size of the support is recovered as follows. Again, by using Dankin's theorem it holds
   \begin{align*}
       \frac{1}{2}\frac{d}{dt}\mathcal{D}^2(t) = \max_{i,j \in \Pi_2(t)}\left\langle \frac{d}{dt}(x_i(t) - x_j(t)), x_i(t) - x_j(t) \right\rangle
   \end{align*}
 where $\Pi_2(t) \in \{1,\ldots,N\} \times \{1,\ldots,N\}$ represents the nonempty subset of pairs of indices for which the maximum is reached.  Fix arbitrary $p,q \in \Pi_2(t)$. For easier notation, from now on we hide the dependence on time. Notice that for the case of normalized weights it holds
\begin{eqnarray*}
        &\left\langle\frac{d}{dt}(x_{p} - x_{q}),x_{p} - x_{q}\right \rangle \\
        &= -\frac{1}{\sum_{k=1}^N \phi_{p k}}\sum_{j=1}^{N} M_{p j} \phi_{p j}\left\langle x_{p}-x_j, x_{p} - x_{q}\right\rangle \\
   & - \frac{1}{\sum_{k=1}^N \phi_{q k}}\sum_{j=1}^{N} M_{q j} \phi_{q j}\left\langle x_j - x_{q},x_{p} - x_{q}\right\rangle .
        \end{eqnarray*}
By Lemma \ref{lemma:scalarproductinequality}, for all $j \in \{1,\ldots,N\}$ it holds
\begin{align*}
    \left\langle x_{p}-x_j, x_{p} - x_{q}\right\rangle \geq 0 \quad \text{and} \quad \left\langle x_j - x_{q},x_{p} - x_{q}\right\rangle \geq 0
\end{align*}
and therefore, by \eqref{defof:kmin}, it holds
\begin{eqnarray*}
         &\left\langle\frac{d}{dt}(x_{p} - x_{q}),x_{p} - x_{q}\right \rangle\\
        &\leq -\frac{K_{\min}}{N} \left(\sum_{j=1}^{N} M_{p j} \left\langle x_{p}-x_j, x_{p} - x_{q}\right\rangle \right.\\
        &\quad \left.+ \sum_{j=1}^{N} M_{q j}\left\langle x_j - x_{q},x_{p} - x_{q}\right\rangle \right)\leq 0.
\end{eqnarray*}

The statement for the minimum on the real line can be recovered by Lemma \ref{l-reverse}.
\end{proof}

\subsection{Dynamics on the real line}

In this section, we study the case $d=1$, i.e. in which the configuration space is the real line $\R$. In this setting, ordering is clearly a major advantage.

We introduce a useful operator, that is $\psi(\alpha,z,\tau)$. The idea is that, given an initial configuration in which the position of all agents has a value larger than $\alpha$, the quantity $\psi(\alpha,z,\tau)$ represents a lower barrier for the position at time $\tau$ of a particle starting at $z$. 
\begin{lemma}\label{l-psi}
Define 
\begin{eqnarray}
\label{e-psi}    
\psi(\alpha,z,\tau)&:=&\alpha+e^{-K_{\max}\tau}(z-\alpha).
\end{eqnarray}
Let $x(t) := \{x_1(t),\ldots,x_N(t)\} $  be a solution of \eqref{e-ODE} on $\R$ that satisfies $x_j(\theta)\geq \alpha$ for all $j\in\{1,\ldots,N\}$ for some $\theta >0$. If there exists an index $i$ and $\tau^*\in [0,T]$ such that $x_i(\theta+\tau^*)\geq \psi(\alpha,z,\tau^*)$, then
\begin{equation}
\label{e-stimapsi}  x_i(\theta+\tau)\geq \psi(\alpha,z,\tau)~~~\mbox{~for all }\tau\in[\tau^*,T].
\end{equation}
\end{lemma}
\begin{proof} Define $\tau:=t-\theta$ and for $t > \theta$ compute 
    \begin{eqnarray*}
        &&\partial_t( x_i(t)-\psi(\alpha,z,t-\theta))=\\
        &&\frac{\lambda_i}{N}\sum_{x_j(t)\leq \psi(\alpha,z,\tau)} M_{ij}(t)\phi_{ij}(t)(x_j(t) - x_i(t))+\\
        &&\frac{\lambda_i}{N}\sum_{x_j(t)> \psi(\alpha,z,\tau)} M_{ij}(t)\phi_{ij}(t)(x_j(t) - x_i(t))\\
        &&+ K_{\max}e^{-K_{\max}\tau}(z-\alpha).
        \end{eqnarray*}
        In the first term, we write 
        \begin{eqnarray*}
        &&\hspace{-5mm} x_j(t) - x_i(t)\geq (\alpha -\psi(\alpha,z,\tau))+(\psi(\alpha,z,\tau)-x_i(t))\\
        && = -e^{-K_{\max}\tau}(z-\alpha)+(\psi(\alpha,z,\tau)-x_i(t)).
        \end{eqnarray*}
    
        Here we used Proposition \ref{prop: contract of support}, since $$x_j(t)\geq \gamma_{\min}(t)\geq \gamma_{\min}(\theta)\geq \alpha.$$
        We then estimate 
        \begin{eqnarray*}
        &&\partial_t( x_i(t)-\psi(\alpha,z,t-\theta))\geq\\
        &&\frac{\lambda_i}{N}\sum_{x_j(t)\leq \psi(\alpha,z,\tau)} M_{ij}(t)\phi_{ij}(t) (-e^{-K_{\max}\tau}(z-\alpha))+\\
        &&\frac{\lambda_i}{N}\sum_{x_j(t)\leq \psi(\alpha,z,\tau)} M_{ij}(t)\phi_{ij}(t)
        (\psi(\alpha,z,\tau)-x_i(t)))+\\
        &&\frac{\lambda_i}{N}\sum_{x_j(t)> \psi(\alpha,z,\tau)} M_{ij}(t)\phi_{ij}(t)(\psi(\alpha,z,\tau) - x_i(t))\\
    &&+K_{\max}e^{-K_{\max}\tau}(z-\alpha)\geq K_{\max} (-e^{-K_{\max}\tau}(z-\alpha))\\
      && +\frac{\lambda_i}{N}\sum_{j=1}^N M_{ij}(t)\phi_{ij}(t)
        (\psi(\alpha,z,\tau)-x_i(t)))+\\
    &&K_{\max}e^{-K_{\max}\tau}(z-\alpha)=\\
        &&a(t)(x_i(t)-\psi(\alpha,z,\tau)),
                \end{eqnarray*}
                where $$a(t):=-\frac{\lambda_i}{N}\sum_{j=1}^{N}M_{ij}(t)\phi_{ij}(t).$$
              
Recall that $x_i(\theta+\tau^*)-\psi(\alpha,z,\tau^*)\geq 0$. We now apply Gr\"onwall's inequality on $x_i(t)-\psi(\alpha,z,t-\theta)$ with $t\in[\theta+\tau^*,\theta+T]$. It ensures
\begin{eqnarray*}
&&x_i(\theta+\tau)-\psi(\alpha,z,\tau)\geq \\
&&e^{\int_{\tau^*}^\tau a(t)\,dt}\left(x_i(\theta+\tau^*)-\psi(\tau^*)\right)\geq 0.
\end{eqnarray*}
This proves \eqref{e-stimapsi}.
\end{proof}

\section{Proof of Theorem \ref{thm:main}} \label{section: proof of main}
%%%%%%%%%%%  Parte di Francesco  $$$$$$$$$$$$$$$$
In this section, we prove Theorem \ref{thm:main}. We first prove it for the 1-dimensional case, then for a general dimension $d>1$.

\subsection{Proof in $\R$}

\label{s-proof1}
In this section, we prove Theorem \ref{thm:main} when the configuration space is $\R$.

Define the sequence $T_n:=nT$. Since the sequence
\begin{equation}
    \{x_1(T_n),\ldots, x_N(T_n)\}\in \R^N \label{e-x1xN}
\end{equation}
is bounded, due to Proposition \ref{prop: contract of support}, there exists as subsequence of $T_n$, that we do not relabel, such that \eqref{e-x1xN} is converging to some $\{x_1^*,\ldots,x_N^*\}$. Since the number of permutations of $N$ agents is finite, there exists a subsequence of $T_n$, that we do not relabel, for which the order of particles is preserved. With no loss of generality, for a single relabeling, we assume 
\begin{eqnarray}\label{e-ordineTn}
    x_1(T_n) \leq x_2(T_n) \leq \ldots \leq  x_N(T_n)
\end{eqnarray}
for all $n \in \mathbb{N}$. This implies 
\begin{eqnarray}\label{ordering of the xi stars}
    x_1^* \leq x_2^* \leq \ldots \leq x_N^*.
\end{eqnarray}
% This also ensures that, for all $\tau\in[0, T]$ and $j\geq 2$ it holds
%     \begin{eqnarray}
%         \L_2^{T_n,T} \leq \L_j^{T_n,T} \leq \L_j^{T_n,\tau}\leq  x_j(T_n+\tau).\label{e-catenaL}
%     \end{eqnarray}
%     The first inequality is \eqref{order xtilde(L,R)L(T)}. The second is  Lemma \ref{Lemma 3.1}, since $\L_j^{T_n,t}$ is decreasing with respect to $t$. The third is \eqref{upper and lower bound on xk inner}.

We now prove that it holds $x_1^*=x_N^*$. By contradiction, from now on we assume 
\begin{equation}
x_1^*<x_N^* \label{e-cond-interessante}.
\end{equation}

We observe that for all $i,j \in \{1,\ldots,N\}$ it holds $\lambda_i \phi_{ij}\geq \phi_{\min}$ in the case of fixed weights and 
$\lambda_i \phi_{ij}\geq \frac{\phi_{\min}}{\phi_{\max}}$ in the case of rescaled weights. It then holds 
\begin{eqnarray}\label{ineq:lambdai phiij geq kmin}
\lambda_i \phi_{ij}\geq K_{\min} \quad \forall ~ i,j \in \{1,\ldots,N\}   
\end{eqnarray}
where $K_{\min}$ is defined by \eqref{defof:kmin}. We now choose 
\begin{eqnarray}\label{defof: eta}
\eta:=\frac{e^{-K_{\max}T}(x_N^*-x_1^*)}4 \frac{K_{\min} \mu }{N(1+K_{\max}e^{K_{\max}T}T)}.    
\end{eqnarray}
Observe that $K_{\min}\leq K_{\max}$ and $\mu\leq T$ ensure
\begin{eqnarray}\label{e-stima-eta}
\eta\leq \frac{e^{-K_{\max}T}(x_N^*-x_1^*)}4.    
\end{eqnarray}

We recall that $\lim_{n\to+\infty} x_1(T_n)=x_1^*$ and $\lim_{n\to+\infty} x_N(T_n)=x_N^*$. We then choose $k\in\mathbb{N}$ such that it holds $x_1(T_k)\in (x_1^*-\eta,x_1^*+\eta)$ and $x_N(T_k)\in (x_N^*-\eta,x_N^*+\eta)$.

{\bf Claim 1.} It holds $x_1(T_k)\in (x_1^*-\eta,x_1^*]$. Similarly, it holds $x_N(T_k)\in [x_N^*,x_N^*+\eta)$. 

We prove the claim, by contradiction. If $x_1(T_k)>x_1^*$, i.e. $x_1(T_k)\geq x_1^*+\epsilon$ for some $\epsilon>0$, then $x_i(T_k)\geq x_1^*+\epsilon$ for all $i$ due to \eqref{e-ordineTn}. This implies hence $\gamma_{\min}(T_k)\geq x_1^*+\epsilon$. By Proposition \ref{prop: contract of support}, this in turn implies $\gamma_{\min}(t)\geq x_1^*+\epsilon$ for all $t\geq T_k$. In particular, it holds $x_1(T_n)\geq \gamma_{\min}(T_n)\geq x_1^*+\epsilon$. Thus, it holds $\lim_{n\to\infty} x_1(T_n)\geq x_1^*+\epsilon$. This contradicts the condition $\lim_{n\to\infty} x_1(T_n)= x_1^*$.

 The condition $x_N(T_k)\in [x_N^*,x_N^*+\eta)$ can be easily recovered, by using Proposition \ref{prop: contract of support}, and we have thus proved Claim 1.

We now define
\begin{eqnarray*}
y_0&:=&x_1^*-\eta+2\eta e^{K_{\max}T}\\
\psi^*(\tau)&:=&\psi(x^*_1-\eta,y_0,\tau),
\end{eqnarray*}
where $\psi$ is defined by \eqref{e-psi} and compute
\begin{eqnarray}\label{e-psi-star}
\psi^*(\tau)&=&x_1^*-\eta+2\eta e^{K_{\max}(T-\tau)}.
\end{eqnarray}

We consider the trajectory of any agent $x_i(T_k+\tau)$ with $\tau\in[0,T]$. We have two cases:

\noindent {\bf Case 1.} There exists $\tau^*\in [0,T]$ such that $$x_i(T_k+\tau^*)\geq \psi^*(\tau^*).$$ 
    By Lemma \ref{l-psi}, this implies 
    $$x_i(T_k+T)\geq \psi^*(T)=x_1^*+\eta.$$

\noindent {\bf Case 2.} It holds $x_i(T_k+\tau)<\psi^*(\tau)$ for all $\tau\in[0,T]$. Here, we observe that $x_N(T_k)\geq x_N^*$, hence Lemma \ref{l-psi} implies 
\begin{eqnarray*}
x_N(T_k+\tau)\geq \psi(x^*_1-\eta,x_N^*,\tau)    
\end{eqnarray*}
for all $\tau\in[0,T]$.

Recall that $\psi$ is given by \eqref{e-psi} and $\psi^*$ is given by \eqref{e-psi-star}. We define $\sigma:=t-T_{k}$,  and for $t \in [T_k, T_k + T]$ compute 
    \begin{eqnarray*}
        &&\dot x_i(t)=\\
        &&\frac{\lambda_i}{N}\sum_{x_j(t)\leq \psi^*(\sigma),j\neq N} M_{ij}(t)\phi_{ij}(t)(x_j(t) - x_i(t))+\\
        &&\frac{\lambda_i}{N}\sum_{x_j(t)> \psi^*(\sigma),j\neq N} M_{ij}(t)\phi_{ij}(t)(x_j(t) - x_i(t))+\\
        &&\frac{\lambda_i}{N}M_{iN}(t)\phi_{iN}(t)(x_N(t) - x_i(t))\geq
        \end{eqnarray*}
    \begin{eqnarray*}
         &&\frac{\lambda_i}{N}\sum_{x_j(t)\leq \psi^*(\sigma),j\neq N} M_{ij}(t)\phi_{ij}(t)(x_1^*-\eta - \psi^*(\sigma))+\\
        &&\frac{\lambda_i}{N}\sum_{x_j(t)> \psi^*(\sigma),j\neq N} M_{ij}(t)\phi_{ij}(t)(\psi^*(\sigma) - x_i(t))+\\
        &&\frac{\lambda_i}{N}M_{iN}(t)\phi_{iN}(t)(\psi(x_1^*-\eta,x_N^*,\sigma) - \psi^*(\sigma))\\
         &&\geq \frac{\lambda_i}{N}\sum_{j=1}^N M_{ij}(t)\phi_{ij}(t)(-2\eta e^{K_{\max}(T-\sigma)})+0+\\
        &&\frac{\lambda_i}{N}M_{iN}(t)\phi_{iN}(t)
e^{-K_{\max}\sigma}(x_N^*-x_1^*+\eta - 2\eta e^{K_{\max}T})\\
        &&\geq-2\eta K_{\max}e^{K_{\max}T}+\\
        &&
        \frac{\lambda_i}{N}M_{iN}(t)\phi_{iN}(t)
e^{-K_{\max}\sigma}\left(x_N^*-x_1^*+0- \frac{x_N^*-x_1^*}2\right).
        \end{eqnarray*}
        where in the last inequality we have used \eqref{e-stima-eta}. By using \eqref{ineq:lambdai phiij geq kmin}, we finally find
        $$\dot x_i(t)\geq -2\eta K_{\max}e^{K_{\max}T}+
        \frac{K_{\min}}{N} M_{iN}(t) e^{-K_{\max}T}\frac{x_N^*-x_1^*}2.$$
        
        By integration in $[T_k, T_k+T]$ and using \eqref{defof:PEgeneral}, it holds
        \begin{eqnarray*}
            &&x_i(T_k+T)\geq x_i(T_k)-2\eta K_{\max}e^{K_{\max}T}T+\\
            &&\frac{K_{\min}}{2N} \mu e^{-K_{\max}T}(x_N^*-x_1^* )\geq\\
            &&x_1^*-\eta-2\eta K_{\max}e^{K_{\max}T}T+2\eta (1+K_{\max}e^{K_{\max}T}T)\\
            &&=x_1^*+\eta.
        \end{eqnarray*}

Here, inequality $x_i(T_k)\geq x_1^*-\eta$ comes from \eqref{e-ordineTn} and Claim 1. The last equality comes from the definition \eqref{defof: eta}.

By merging the two cases, we proved that $x_i(T_k+T)\geq x_1^*+\eta$ for all $i$. By Proposition \ref{prop: contract of support}, this implies that $x_i(t)\geq x_1^*+\eta$ for all $t\geq T_k+T$. In particular, this implies $x_1(t)\geq x_1^*+\eta$, hence $\lim_{n\to+\infty} x_1(T_n)\geq x_1^*+\eta>x_1^*$. This is a contradiction. 

Then, condition \eqref{e-cond-interessante} is false and it holds $x_1^*=x_N^*$. We now prove that this implies consensus. Indeed, for each $t>0$ choose $T_n$ such that $T_n\leq t \leq T_{n+1}$. Observe that, due to Proposition \ref{prop: contract of support}, for each $x_i$ it holds
$$x_1(T_n)=\gamma_{\min}(T_n)\leq \gamma_{\min}(t)\leq x_i(t)$$
and similarly $x_i(t)\leq x_N(T_n)$. This implies
\begin{eqnarray*}
    &&\lim_{t\to+\infty} x_i(t)-x_j(t)\leq \lim_{n\to+\infty}x_N(T_n)-x_1(T_n)=\\
    &&x_N^*-x_1^*=0.
    \end{eqnarray*}
    Since $i,j$ are arbitrary, it also holds $\lim_{t\to+\infty} x_j(t)-x_i(t)\leq 0$. Then $\lim_{t\to+\infty} x_i(t)-x_j(t)= 0$. Again by arbitrariness of indexes, the result is proved.

\subsection{Proof for any dimension}

In this section, we prove Theorem \ref{thm:main} on $\R^d$ for any $d>1$. The idea is to show that the dynamics can be projected on a line, and to use the result of Section \ref{s-proof1} on such line.

The key observation is the following. Fix two vectors $x_0,v\in \R^d$ and take a solution $\{x_i(t)\}$ of \eqref{e-ODE} in $\R^d$. Define the projected solution as 
$$y_i(t):=(x_i(t)-x_0)\cdot v.$$

In general, it is clear that $y_i(t)$ is not a solution of \eqref{e-ODE}, since it holds $$\phi_{ij}(t)=\phi(|x_i(t) - x_j(t)|)\neq \phi(|y_i(t) - y_j(t)|).$$

Yet, a careful look to the proof in Section \ref{s-proof1} shows that the key properties ensuring the result are Proposition \ref{prop: contract of support} and estimates for the interaction kernels encoded in \eqref{sum of kernels less than eta max}. In higher dimension, the fact that Proposition \ref{prop: contract of support} holds also implies that $\phi_{\min},\phi_{\max}>0$ computed for the $x_i$ variables provide corresponding bounds for the interaction of the $y_i$. As a consequence, Theorem \ref{thm:main} also holds for the variables $y_i$, for any choice of fixed $x_0,v\in \R^d$.

We now prove the theorem. Consider the sequence of times $T_n:=nT$. By Proposition \ref{prop: contract of support}, the sequence $\{x_1(T_n),\ldots,x_N(T_n)\}$ is bounded. By passing to a subsequence $T_n$, that we do not relabel, we have that the sequence is converging to some $\{x_1^*,\ldots,x_N^*\}$. Choose one among the maximizers of $|x_i^*-x_j^*|$. With no loss of generality, we assume that it is realized by $|x_1^*-x_N^*|$. Choose $x_0=x_1^*$ and $v=x_N^*-x_1^*$. It is crucial to observe that they are fixed vectors. Define the variables $y_i(t)=(x_i(t)-x_1^*)\cdot v$ and apply Theorem \ref{thm:main} on the real line. It then holds $\lim_{t\to+\infty} y_i(t)=y^*$ for some common $y^*\in\R$. By choosing $y_1(t)=(x_i(t)-x_1^*)\cdot v$, it holds $\lim_{t\to+\infty} y_1(t)=0$, hence $y^*=0$. By choosing $y_n(t)=(x_n(t)-x_1^*)\cdot v$, it holds 
$$\lim_{t\to+\infty} y_n(t)=\lim_{t\to+\infty} (x_n(t)-x_1^*)\cdot v=|v|^2=y^*=0.$$
This implies $v=0$, i.e. $x_1^*=x_N^*$. By recalling that $x_1,x_N$ were chosen as maximizers of $|x_i^*-x_j^*|$, this implies $x_i^*=x_1^*$ for all $i$, i.e. $\lim_{n\to+\infty} x_i(T_n)=x_1^*$. We rewrite it as follows: for all $\epsilon>0$, there exists $k\in\mathbb{N}$ such that $x_i(T_n)\in B_\epsilon (x_1^*)$ for all $n\geq k$. By Proposition \ref{prop: contract of support} and recalling that $B_\epsilon (x_1^*)$ is convex, we have that for all $t\geq T_k$ it holds $x_i(t)\in B_\epsilon (x_1^*)$. This is equivalent to $\lim_{t\to +\infty} x_i(t)=x_1^*$ for all $i$, that implies consensus.

\section{Simulations} \label{s-sim}

In this section, we provide some simulations for the dynamics given by \eqref{e-ODE}. For simplicity, we always study the dynamics of $N=10$ agents, with $\lambda_i=1$ fixed and $T=1$ in the PE condition. 

We first consider $N=10$ agents with initial positions randomly chosen from the uniform probability distribution on $[0,1]$. For a given initial position, we show in Figure \ref{f-traj} four solutions of \eqref{e-ODE} with $M_{ij}$ satisfying the PE condition with $\mu=1,0.6,0.3,0.1$, respectively. It is clear that convergence always holds, in accordance with Theorem \ref{thm:main}. Yet, the rate of convergence decreases with the decrease of $\mu$.

\begin{figure}
    \centering
\includegraphics[width=8cm]{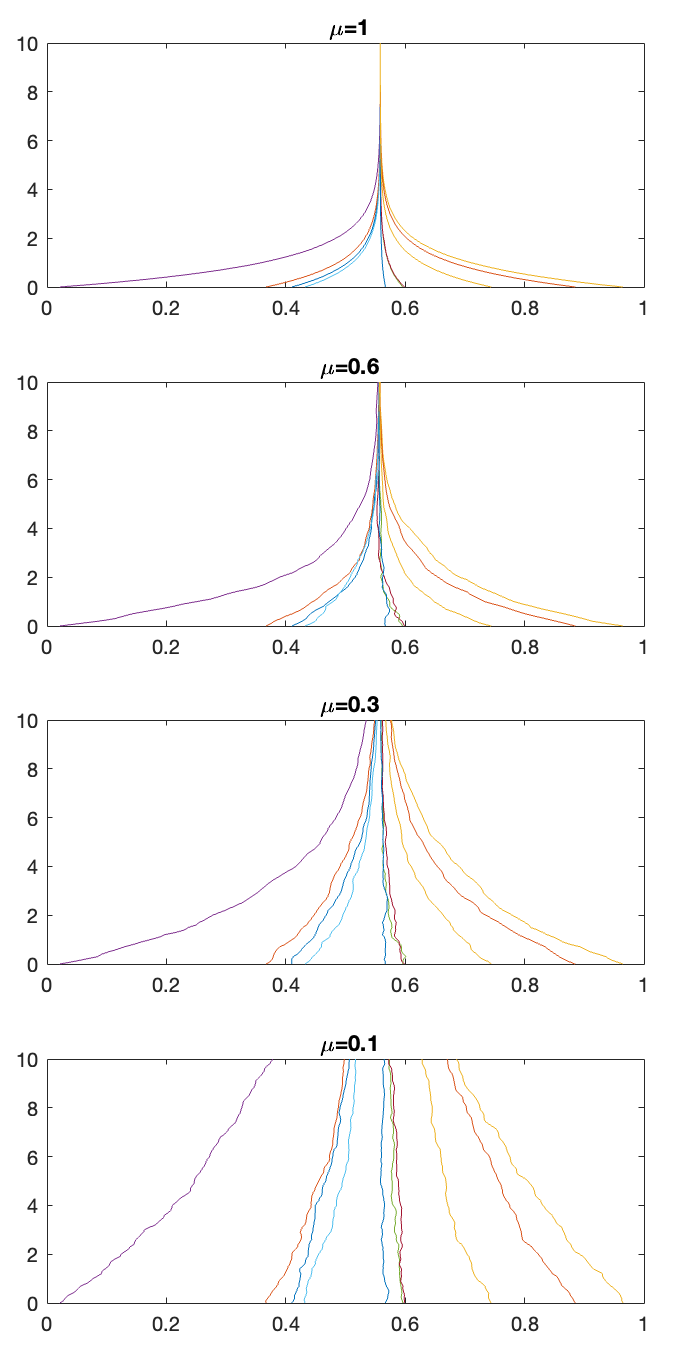}
    \caption{Solutions of \eqref{e-ODE} with varying $\mu$.}
    \label{f-traj}
\end{figure}

We investigate more in detail the rate of convergence as a function of $\mu$. We consider $10^3$ different initial configurations and, for each of them and each value of $\mu$, we compute the time to reach a diameter of $10^{-2}$. For each $\mu$, we then consider the average of such times, that can be seen as a good measure of the rate of convergence. The result is provided in Figure \ref{f-rate}, that is a log-log graph of the value of $\mu$ and the corresponding average time to convergence. The fact that the graph is linear shows that the average behavior is very regular: the PE condition becomes a uniform reduction of the interaction $M_{ij}(t)\phi$ to $\mu\phi$. We aim to deepen the comprehension of this phenomenon in a future work.

\begin{figure}
    \centering
\includegraphics[width=8cm]{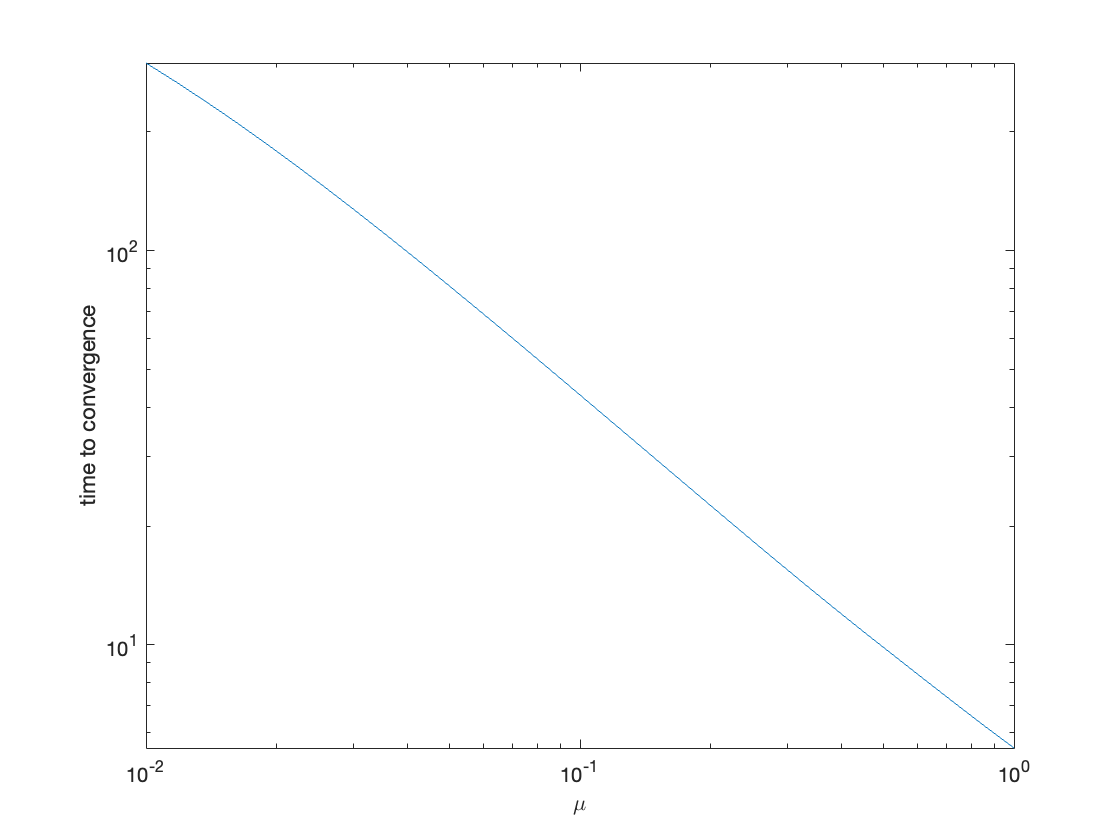}
    \caption{The average time to convergence as a function of $\mu$.}
    \label{f-rate}
\end{figure}

\section{Conclusions and future directions}\label{section: conclusion}

In this article, we have proved that the condition of Persistence Excitation \eqref{defof:PEgeneral} %\nameref{defof:PEgeneral} 
is sufficient to ensure consensus of a cooperative multi-agent system.

The natural extension of this result would be to prove a similar result for second-order systems, such as those describing velocity alignement and flocking \cite{cucker2002mathematical}. Yet, the natural difficulty in this project is the fact that our proof here does not provide a rate of convergence towards the consensus. In second-order systems, this corresponds to a lack of knowledge of the expansion of the support, undermining the certainty of having an interaction function bounded from below by a strictly positive constant, i.e. {\bf (H2)}.

Our project is then to provide quantitative estimates for the rate of convergence, as a first tool towards a more general applicability of the theory. The simulations provided here are a first step towards this direction.

\bibliographystyle{ieeetr}
\bibliography{references.bib}
% \printbibliography%[heading=bibintoc]

\end{document}

%% file: new_commands.tex
\newtheorem{theorem}{Theorem}[section]
\newtheorem{lemma}{Lemma}[section]

\newtheorem{proposition}{Proposition}[section]
\newtheorem{definition}{Definition}[section]

\newtheorem{remark}{Remark}[section]

\renewcommand{\theenumii}{\theenumi.\arabic{enumii}.}